\documentclass{amsart}

\usepackage{
    mathtools,
    amssymb,
    amsthm,
    enumitem,
    tikz-cd,
    booktabs,
    hyperref,
}
\usepackage[foot]{amsaddr}

\hypersetup{linktoc=all}
\numberwithin{equation}{section}
\numberwithin{table}   {section}

\newtheorem{thm}        {Theorem}[section]
\newtheorem{lem}  [thm] {Lemma}
\newtheorem{cor}  [thm] {Corollary}
\newtheorem{prop} [thm] {Proposition}
\newtheorem{prob} [thm] {Problem}
\theoremstyle{definition}
\newtheorem{dfn}  [thm] {Definition}
\newtheorem{rmk}  [thm] {Remark}
\newtheorem{nts}  [thm] {Notation}

\DeclarePairedDelimiterX\set[1]\lbrace\rbrace{\,\def\given{\;\delimsize\vert\;}#1\,}
\newcommand{\nat}
    {\mathbb{N}}
\newcommand{\Fp}
    {\ensuremath{\mathbb{F}_p}}
\newcommand{\alt}[1]
    {\ensuremath{\mathfrak{A}_{#1}}}
\newcommand{\order}[1]
    {\ensuremath{\lvert{#1}\rvert}}
\renewcommand{\radical}[1]
    {\ensuremath{\operatorname{Rad}\left({#1}\right)}}
\newcommand{\rad}[2]
    {\ensuremath{\operatorname{Rad}^{#1}\left({#2}\right)}}
\newcommand{\soc}[2]
    {\ensuremath{\operatorname{Soc}^{#1}\left({#2}\right)}}
\newcommand{\loewy}[1]
    {\ensuremath{\ell\ell\left(#1\right)}}
\newcommand{\z}[1]
    {\ensuremath{Z(#1)}}

\newcommand{\ordprod}
    {\sideset{}{'}\prod}
\newcommand{\ordprodsub}[1]
    {\ordprod_{\substack{#1}}}
\newcommand{\sumsub}[1]
    {\sum_{\substack{#1}}}
\newcommand{\zs}[2]
    {\ensuremath{ZS^{#1}(#2)}}
\newcommand{\gij}
    {g_{ij}}
\newcommand{\mij}
    {{m_{ij}}}

\begin{document}

\title
    [Central elements of the Jennings basis]
    {Central elements of the Jennings basis and\\ certain Morita invariants}
\date{\today}
\author[T. Sakurai]
    {\href{https://orcid.org/0000-0003-0608-1852}
    {Taro Sakurai}}
\address{
    Department of Mathematics and Informatics,
    Graduate School of Science,
    Chiba University,
    1-33,
    Yayoi-cho,
    Inage-ku,
    Chiba-shi,
    Chiba,
    263-8522 Japan
}
\email{tsakurai@math.s.chiba-u.ac.jp}

\begin{abstract}
    From Morita theoretic viewpoint, computing Morita invariants is important.
    We prove that the intersection of the center and the \( n \)th (right) socle
    \( ZS^n(A) := Z(A) \cap \operatorname{Soc}^n(A) \)
    of a finite-dimensional algebra \( A \) is a Morita invariant;
    This is a generalization of important Morita invariants --- the center \( Z(A) \) and the Reynolds ideal \( ZS^1(A) \).

    As an example, we also studied \( ZS^n(FG) \) for the group algebra \( FG \) of a finite \( p \)-group \( G \) over a field \( F \) of positive characteristic \( p \).
    Such an algebra has a basis along the socle filtration, known as the Jennings basis.
    We prove certain elements of the Jennings basis are central and hence form a linearly independent set of \( ZS^n(FG) \).
    In fact, such elements form a basis of \( ZS^n(FG) \) for every integer \( 1 \le n \le p \) if \( G \) is powerful.
    As a corollary we have \( \operatorname{Soc}^p(FG) \subseteq Z(FG) \) if \( G \) is powerful.
\end{abstract}

\maketitle

\setcounter{tocdepth}{1}
\tableofcontents

\section{Introduction}

\noindent 
From Morita theoretic viewpoint,
    computing Morita invariants is important to distinguish algebras that are not Morita equivalent.
We prove that the intersection of the center and the \( n \)th (right) socle
\begin{equation}
    \label{eq: ZS}
    \zs{n}{A} := \z{A} \cap \soc{n}{A}
\end{equation}
is Morita invariant for a finite-dimensional algebra \( A \) (Theorem~\ref{thm: invariant}).
This is a generalization of important Morita invariants
    --- the center \z{A} and the Reynolds ideal \zs{1}{A}.
The other way of generalization is known as the K\"ulshammer ideals or the generalized Reynolds ideals for a finite-dimensional symmetric algebra;
    For more details we refer the reader to the survey by Zimmermann \cite{Zim11}.

The centers and the Reynolds ideals are particularly interesting for finite group algebras.
Let \( G \) be a finite group and \( F \) an algebraically closed field \( F \) of positive characteristic \( p \).
First,
    the dimension of the center             \z{FG}  equals
        the number of irreducible ordinary characters \( k(G) \) and
    the dimension of the Reynolds ideal \zs{1}{FG}  equals
        the number of irreducible modular  characters \( \ell(G) \).
Next,
    the         conjugacy class sums form a basis of the center and
    the \( p \)-regular section sums form a basis of the Reynolds ideal \cite{Kue81}.
Moreover,
    Okuyama \cite{Oku81} proved
    \begin{equation}
        \label{eq: okuyama}
        \dim \zs{2}{FG} = \ell(G) + \sum \dim \operatorname{Ext}_A^1(S, S)
    \end{equation}
    where the sum is taken over a complete set of simple \( FG \)-modules.
(In fact, he proved that for a block of a finite group algebra.)
See also \cite[Theorem~2.1]{Kos16} which is written in English.
These are summarized in Table~\ref{tbl: known}.

\begin{table}[h]
    \centering
    \caption{What is known about \( \zs{n}{FG} = \z{FG} \cap \soc{n}{FG} \).}
    \label{tbl: known}
    \begin{tabular}{lll}
        \toprule
            & dimension
                & basis
        \\
            & (representation-theoretic)
                & (group-theoretic)
        \\

        \cmidrule(lr){2-3}
        \z{FG}
            & \( k(G) \)
                & conjugacy class sums
        \\
        \zs{n}{FG}
            & unknown
                & unknown
        \\
        \zs{2}{FG}
        & \( \ell(G) + \sum \dim \operatorname{Ext}^1(S, S) \)
            & unknown\footnotemark[2]
        \\
        \zs{1}{FG}
            & \( \ell(G) \)
                & \( p \)-regular section sums
        \\
        \bottomrule
    \end{tabular}
\end{table}
\footnotetext[2]{Except finite \( p \)-groups; see Remark~\ref{rmk: corner cases}.}

As \zs{n}{FG} is a generalization of these,
    we want to know what is the dimension and what a basis can be (Problem~\ref{prob}).
One of manageable examples to compute socle series is the case for a finite \( p \)-group \( G \).
Jennings constructed group-theoretically a basis of \( FG \) along the radical filtration (Theorem~\ref{thm: Jennings}) and
it follows that radical series coincides with socle series (Theorem~\ref{thm: Jennings-Brauer}).
Such a basis is known as the Jennings basis (Definition~\ref{dfn: Jennings basis}).
To study \zs{n}{FG} let us ask a question:
    \begin{center} When is an element of the Jennings basis central? \end{center}
We prove that certain elements of the Jennings basis are central and hence form a linearly independent set of \zs{n}{FG} (Theorem~\ref{thm: main}).
In fact, such elements form a basis of \zs{n}{FG} for every integer \( 1 \le n \le p \) if \( G \) is powerful (Theorem~\ref{thm: powerful}).
As a corollary we have \( \soc{p}{FG} \subseteq \z{FG} \) if \( G \) is powerful (Corollary~\ref{cor: soc p}).
Proofs are given in Section~\ref{sec: proofs}.
Section~\ref{sec: examples} is devoted to illustrate our results by the first non-trivial examples.

\section{Morita invariants}

\label{sec: invariant}
H\'ethelyi et al. \cite[Corollary~5.3]{HHKM05} proved that the K\"ulshammer ideals are Morita invariants.
(Later Zimmermann \cite{Zim07} proved that those are derived invariants as well.)
Inspired by the proof we prove that \zs{n}{A}, ideals of the center, are Morita invariants.
The \( n \)th Jacobson radical of \( A \) will be denoted by \rad{n}{A}.

\begin{lem}
    \label{lem: Loewy length}
    The Loewy length \loewy{A} of a finite-dimensional algebra \( A \) over a field
    can be expressed by \zs{n}{A} as
    \begin{equation*}
        \loewy{A} = \min\set{ n \in \nat \given \zs{n}{A} = \z{A} }.
    \end{equation*}
\end{lem}

\begin{proof}
    Clear.
\end{proof}

\begin{lem}
    \label{lem: rad idemp}
    Let \( A \) be a finite-dimensional algebra over a field and
    \( e \in A \) a full idempotent.
    Then \( \rad{n}{eAe} = e\rad{n}{A}e \) for every \( n \in \nat \).
\end{lem}

\begin{proof}
    Since there is a canonical lattice isomorphism from ideals of \( A \) to those of \( eAe \)
    which preserves multiplication of ideals,
    the claim follows.
    (See \cite[Theorem~21.11(2)]{Lam91}.)
\end{proof}

\begin{thm}
    \label{thm: invariant}
    Let \( A \) and \( B \) be Morita equivalent finite-dimensional algebras over a field.
    Then there is an algebra isomorphism \( \z{A} \to \z{B} \)
    mapping \zs{n}{A} onto \zs{n}{B} for every \( n \in \nat \).
    In particular, \zs{n}{A} are Morita invariants.
\end{thm}

\begin{proof}
    It suffices to prove the case for \( A \) and its basic algebra \( B:= eAe \)
    because Morita equivalent basic algebras are isomorphic.
    Since the basic idempotent \( e \in A \) is full,
    there exist \( u_k, v_k \in A \) such that
    \begin{equation}
        \label{eq: unit decomp}
        \sum_k u_k e v_k = 1.
    \end{equation}
    Then define algebra homomorphisms
    \begin{equation*} \begin{tikzcd}
        \z{A} \arrow[r, "\phi", shift left]
        &
        \z{B} \arrow[l, "\psi", shift left]
    \end{tikzcd} \end{equation*}
    by
    \( \phi(a) = eae \)              for \( a \in \z{A} \) and
    \( \psi(b) = \sum_k u_k b v_k \) for \( b \in \z{B} \).
    These are well-defined and mutually inverse.
    By \eqref{eq: unit decomp} and Lemma~\ref{lem: rad idemp} we have
    \begin{equation*}
        \phi\big(\zs{n}{A}\big) \subseteq \zs{n}{B}
        \quad \text{and} \quad
        \psi\big(\zs{n}{B}\big) \subseteq \zs{n}{A},
    \end{equation*}
    and the proof completes.
\end{proof}

\begin{rmk}
    In general, \zs{n}{A} are not derived invariants.
    Let \( F \) be an algebraically closed field of characteristic two,
    \alt{n} the alternating group of degree \( n \), and
    \( e_0 \) the principal block idempotent of \( F\alt{5} \).
    Then
    \(
        \loewy{F\alt{4}} = 3 \neq 5 = \loewy{e_0 F\alt{5}}.
    \)
    In particular, by Lemma~\ref{lem: Loewy length}, we have
    \begin{equation*}
        \dim \zs{3}{    F\alt{4}}
        \neq
        \dim \zs{3}{e_0 F\alt{5}}
    \end{equation*}
    although \( F\alt{4} \) and \( e_0 F\alt{5} \) are derived equivalent.
    (cf. the Brou\'e abelian defect conjecture.)
\end{rmk}

Now let us study how \zs{n}{A} behaves with field extensions.

\begin{lem}
    \label{lem: ext rad soc}
    Let \( A \) be a finite-dimensional algebra over a field \( F \) and suppose \( A/\radical{A} \) is separable.
    Then for every field extension \( E/F \) and \( n \in \nat \) we have the following.
    \begin{enumerate}[label={\upshape(\roman*)}]
        \item
            \( \rad{n}{A \otimes_F E} = \rad{n}{A} \otimes_F E \).
            \label{item: radical extension}
        \item
            \( \soc{n}{A \otimes_F E} = \soc{n}{A} \otimes_F E \).
            \label{item: socle extension}
    \end{enumerate}
\end{lem}

\begin{proof}
    \ref{item: radical extension}:
    The proof is by induction on \( n \) and \cite[Lemma 2.5.1(ii)]{NT89}.

    \ref{item: socle extension}:
    Let \( a \in \soc{n}{A \otimes_F E} \) and
    take an \( F \)-basis \( \set{ \varepsilon_\lambda \given \lambda \in \Lambda } \) of \( E \).
    Since
    \(
        A \otimes_F E = \bigoplus A \otimes \varepsilon_\lambda
    \)
    we can express
    \(
        a = \sum a_\lambda \otimes \varepsilon_\lambda
    \)
    for some \( a_\lambda \in A \), which equals zero for all but finitely many \( \lambda \in \Lambda \).
    For every \( z \in \rad{n}{A} \) we have \( z \otimes 1 \in \rad{n}{A \otimes_F E} \) by \ref{item: radical extension}.
    Hence
    \begin{equation*}
        \sum (a_\lambda z) \otimes \varepsilon_\lambda
        =
        a(z \otimes 1)
        =
        0.
    \end{equation*}
    By uniqueness we have \( a_\lambda z = 0 \),
    which implies \( a_\lambda \in \soc{n}{A} \), for all \( \lambda \in \Lambda \).
    We thus get   \( \soc{n}{A} \otimes_F E \subseteq \soc{n}{A \otimes_F E} \).

    The proof for \( \soc{n}{A} \otimes_F E \subseteq \soc{n}{A \otimes_F E} \) is easier.
\end{proof}

\begin{prop}
    \label{prop: ext zs}
    Let \( A \) be a finite-dimensional algebra over a field \( F \) and suppose \( A/\radical{A} \) is separable.
    Then for every field extension \( E/F \) and \( n \in \nat \) we have
    \begin{equation*}
        \zs{n}{A \otimes_F E} = \zs{n}{A} \otimes_F E.
    \end{equation*}
\end{prop}

\begin{proof}
    \begin{align*}
        \zs{n}{A \otimes_F E}
            &= \z{A \otimes_F E} \cap \soc{n}{A \otimes_F E}
                & &
            \\
            &= \big(\z{A} \otimes_F \z{E}\big) \cap \soc{n}{A \otimes_F E}
                & &\text{(By \cite[Lemma~2.4.1(iii)]{NT89}.)}
            \\
            &= \big(\z{A} \otimes_F E\big) \cap \big(\soc{n}{A} \otimes_F E\big)
                & &\text{(By Lemma~\ref{lem: ext rad soc}\ref{item: socle extension}.)}
            \\
            &= \big(\z{A} \cap \soc{n}{A}\big) \otimes_F E
                & &\text{(By \cite[Problem 1.21]{NT89}.)}
            \\
            &= \zs{n}{A} \otimes_F E.
                & &
            \qedhere
    \end{align*}
\end{proof}

\begin{cor}
    \label{cor: reduction}
    Let \( F \) be a field of positive characteristic \( p \) and \( G \) a finite group.
    Then for every \( n \in \nat \) we have
    \begin{equation*}
        \zs{n}{FG} = \zs{n}{\Fp G} \otimes_{\Fp} F.
    \end{equation*}
\end{cor}

\begin{proof}
    Since \( \Fp G/\radical{\Fp G} \) is separable \cite[Lemma~3.1.28]{NT89},
    the claim follows from Proposition~\ref{prop: ext zs}.
\end{proof}

It would be interesting to study the following problem.

\begin{prob}[See also Table~\ref{tbl: known}]
    \label{prob}
    Let \( G \) be a finite group and \( p \) a prime divider of \order{G}.
    Construct bases and describe the dimensions of \zs{n}{\Fp G} for all \( n \in \nat \).
\end{prob}

By Lemma~\ref{lem: Loewy length},
a satisfactory answer to this problem
yields an answer to the Brauer Problem~15 \cite{Bra63}
--- a group-theoretic description of the Loewy length \loewy{FG}.

\begin{rmk}
    Let \( A \) be a block of a finite group algebra over a splitting field and
    \( \{ e_i \} \) a complete set of orthogonal primitive idempotents of \( A \).
    Recently Otokita \cite{Oto16} proved an upper bound
    \begin{equation*}
        \dim \zs{n}{A} \le \sum_i \dim e_i A e_i/e_i \rad{n}{A} e_i
    \end{equation*}
    for every \( n \in \nat \).
\end{rmk}

\section{Jennings theory}

As we proved \zs{n}{A} are Morita invariaits,
we want to determine these for special cases.
Taking Problem~\ref{prob} into account, we hereafter study the group algebra \( FG \)
of a finite \( p \)-group \( G \) over a field \( F \) of positive characteristic \( p \).
In this section, we collect some results of the Jennings theory
for the reader's convenience and to fix our notations.

\begin{dfn}
    \label{dfn: dimension subgroup}
    For \( i \in \mathbb{N} \) we define the \( i \)th \emph{dimension subgroup} (or \emph{Jennigs subgroup}) of \( G \) by
    \begin{equation*}
        D_i := \set{ g \in G \given g - 1 \in \rad{i}{FG} }.
    \end{equation*}
\end{dfn}

\begin{rmk}
    Although the dimension subgroups are defined ring-theoretically,
    these can be computed group-theoretically
    by Theorem~\ref{thm: Jennings-Brauer}\ref{item: M-series}.
\end{rmk}

\begin{lem}
    \label{lem: dimension subgroup}
    \noindent
    \begin{enumerate}[label={\upshape(\roman*)}]
        \item Dimension subgroups are a descending series of characteristic subgroups.
            \label{item: characteristic}
        \item Every successive quotient of the dimension subgroups is an elementary abelian \( p \)-group.
    \end{enumerate}
\end{lem}

\begin{proof}
    See \cite[177--178]{Jen41}.
\end{proof}

We use the following notations throughout the paper.

\begin{nts}
    \label{nts: Jennings}
    Let \( D_i \) be the dimension subgroups of \( G \).
    Set \( t := \min\set{ i \in \nat \given D_i = 1 }  \).
    Then we fix elements \( g_{i1}, \dotsc, g_{ir_i} \in D_i \) such that
    \( \set{ \gij D_{i+1} \given 1 \le j \le r_i } \) form a minimal generating set of
    \( D_i/D_{i+1} \) for every integer \( 1 \le i < t \) satisfying \( D_i > D_{i+1} \).
    We write \( \ordprod \) for the product taken in lexicographic order with respect to indices.
\end{nts}

\begin{thm}[Jennings]
    \label{thm: Jennings}
    For every integer \( n \ge 0 \) we have
    \begin{equation*}
        \label{eq: Jennings}
        \rad{n}{FG}
        =
        \bigoplus F \ordprodsub{1 \le i < t\\ 1 \le j \le r_i} (\gij - 1)^\mij
    \end{equation*}
    where the direct sum is taken for all integers \( 0 \le \mij < p \) satisfying
    \begin{equation*}
        \sumsub{1 \le i < t\\ 1 \le j \le r_i} i \mij \ge n.
    \end{equation*}
\end{thm}

\begin{proof}
    See \cite[Theorem~3.2]{Jen41}.
\end{proof}

\begin{thm}[Jennings, Brauer]
    \label{thm: Jennings-Brauer}
    \noindent
    \begin{enumerate}[label={\upshape(\roman*)}]
        \item The Loewy length \loewy{FG} of \( FG \) equals \( \displaystyle 1 + (p - 1)\sum_{1 \le i < t} ir_i \).
        \item \( FG \) is rigid\textup{:}\\ \( \soc{n}{FG} = \rad{\loewy{FG} - n}{FG} \) for every integer \( 0 \le n \le \loewy{FG} \).
            \label{item: rigid}
        \item \( \displaystyle \soc{n}{FG} = \bigoplus F\ordprodsub{1 \le i < t\\ 1 \le j \le r_i} (\gij - 1)^\mij \) for every integer \( n \ge 0 \)
            where the direct sum is taken for all integers \( 0 \le \mij < p \) satisfying \( \displaystyle \sumsub{1 \le i < t\\ 1 \le j \le r_i} i(p - 1 - \mij) < n \).
            \label{item: soc}
        \item \( D_1 = G \) and \( D_i = (D_{\lceil i/p \rceil})^p[D_{i-1}, G] \) for every integer \( i > 1 \).
            \label{item: M-series}
    \end{enumerate}
\end{thm}

\begin{proof}
    \noindent
    \begin{enumerate}[label={\upshape(\roman*)}]
        \item See \cite[Theorem~3.7]{Jen41}.
        \item See \cite[Proof of Corollary~3.14.7]{Ben91}.
        \item The proof follows from part \ref{item: rigid} and Theorem~\ref{thm: Jennings}.
        \item See \cite[Theorem~5.5]{Jen41}.
        \qedhere
    \end{enumerate}
\end{proof}

\begin{dfn}
    \label{dfn: Jennings basis}
    The basis
    \begin{equation*}
        \set[\Bigg]{ \ordprodsub{1 \le i < t\\ 1 \le j \le r_i} (\gij - 1)^\mij \given 0 \le \mij < p }
    \end{equation*}
    of \( FG \) is said to be the \emph{Jennings basis}.
\end{dfn}

\section{Main theorems}
\label{sec: main}

Recall Notation~\ref{nts: Jennings} in the following.

\begin{thm}
    \label{thm: main}
    Suppose \( s \in \nat \) satisfies \( D_s \ge [G, G] \).
    Then an element of the Jennings basis of the form
    \begin{equation*}
        \ordprod_{\substack{1 \le i < s\\ 1 \le j \le r_i}} (\gij - 1)^\mij
        \ordprod_{\substack{s \le i < t\\ 1 \le j \le r_i}} (\gij - 1)^{p-1}
    \end{equation*}
    are central for every integers \( 0 \le \mij < p \).
    In particular, for every \( n \in \nat \), we have
    \begin{equation}
        \label{eq: main2}
        \zs{n}{FG}
        \supseteq
        \bigoplus F
            \ordprodsub{1 \le i < s\\ 1 \le j \le r_i} (\gij - 1)^\mij
            \ordprodsub{s \le i < t\\ 1 \le j \le r_i} (\gij - 1)^{p-1}
    \end{equation}
    where the direct sum is taken for all integers \( 0 \le \mij < p \) satisfying
    \begin{equation*}
        \sumsub{1 \le i < s\\ 1 \le j \le r_i} i(p - 1 - \mij) < n.
    \end{equation*}
\end{thm}

\begin{rmk}
    Note \( D_2 \ge [G, G] \).
\end{rmk}

\begin{rmk}
    From \eqref{eq: main2}, we have a lower bound
    \begin{equation*}
        \dim \zs{n_s}{FG}
        \ge
        \order{G/D_s}
    \end{equation*}
    where \( \displaystyle n_s := 1 + (p - 1)\sum_{1 \le i < s} ir_i \).
\end{rmk}

See Section~\ref{sec: examples} for concrete examples.
We can show that equality holds in \eqref{eq: main2} for the following class of \( p \)-groups.

\begin{dfn}[Lubotzky-Mann \cite{LM87}]
    A finite \( p \)-group \( G \) is said to be \emph{powerful}
    if \( G^p \ge [G, G] \) and \( p > 2 \),
    or \( G^4 \ge [G, G] \) and \( p = 2 \).
\end{dfn}

\begin{thm}
    \label{thm: powerful}
    If \( G \) is powerful then, for every integer \( 1 \le n \le p \), we have
    \begin{equation*}
        \zs{n}{FG}
        =
        \bigoplus F
            \ordprod_                {1 \le j \le r_1} (g_{1j} - 1)^{m_{1j}}
            \ordprodsub{2 \le i < t\\ 1 \le j \le r_i} (\gij - 1)  ^{p-1}
    \end{equation*}
    where the direct sum is taken for all integers \( 0 \le m_{1j} < p \) satisfying
    \begin{equation*}
        \sum_{1 \le j \le r_1} (p - 1 - m_{1j}) < n.
    \end{equation*}
\end{thm}

\begin{cor}
    \label{cor: soc p}
    If \( G \) is powerful then \( \soc{p}{FG} \subseteq \z{FG} \).
\end{cor}

\begin{rmk}
    Even if \( G \) is powerful, the assertion \( \soc{p+1}{FG} \subseteq \z{FG} \) is false.
    See Subsection~\ref{sec: exponent p^2}.
\end{rmk}

\begin{rmk}
    \label{rmk: Mueller}
    Let \( A \) be a finite-dimensional symmetric algebra over a field.
    In general, it is known that \( \soc{2}{A} \subseteq \z{A} \) if \( A \) is split-local.
    This can be traced back to M\"uller \cite[Proof of Lemma~2]{Mue74}.
    (For a simple proof see, for example, \cite[Lemma~2.2]{BKL17}.)
\end{rmk}

\begin{rmk}
    \label{rmk: corner cases}
    Note that \zs{n}{FG} can be described explicitly for \( n = 1, 2 \).
    \begin{align*}
        \zs{1}{FG}
        &= F \ordprodsub{1 \le i < t\\ 1 \le j \le r_i} (\gij - 1)^{p-1}
        \\
        \zs{2}{FG}
        &= F \ordprodsub{1 \le i < t\\ 1 \le j \le r_i} (\gij - 1)^{p-1}
        \oplus
        \bigoplus_{\mathclap{1 \le s \le r_1}} F
            \ordprod_                {1 \le j \le r_1} (g_{1j} - 1)^{p-1-\delta_{js}}
            \ordprodsub{2 \le i < t\\ 1 \le j \le r_i} (\gij   - 1)^{p-1}
    \end{align*}
    This is due to the Jennings theory and Remark~\ref{rmk: Mueller}.
\end{rmk}

\begin{rmk}
    It is, of course, not true that \zs{n}{FG} is always spanned by a subset of the Jennings basis;
    Any 2-group of maximal class of order 16 yields a minimal counterexample \cite{GAP4, Sak17}.
\end{rmk}

\section{Proofs}
\label{sec: proofs}

\subsection{Proofs of lemmas}

A certain ideal associated to a normal subgroup
plays an important role in the proof of the main theorems.
We prove its properties in the following.

\begin{lem}
    \label{lem: central ideal}
    Let \( G \) be a finite group,
    \( N \unlhd G \), and
    \( F \) a field.
    Let \( N^+ \) denote the sum of all elements of \( N \) in \( FG \).
    If \( N \ge [G, G] \) then
    \(
        FG \cdot N^+
        \subseteq
        \z{FG}
    \).
\end{lem}

\begin{proof}
    Clear.
\end{proof}

\begin{rmk}
    The \( N = [G, G] \) case of Lemma~\ref{lem: central ideal}
    can be found in \cite[Lemma~5]{Wal62}.
\end{rmk}

\begin{lem}
    \label{lem: divisors trick}
    \begin{equation*}
        \ordprodsub{s \le i < t\\ 1 \le j \le r_i} (\gij - 1)^{p-1}
        =
        D_s^+.
    \end{equation*}
\end{lem}

\begin{proof}
    \begin{align*}
        \ordprodsub{s \le i < t\\ 1 \le j \le r_i} (\gij - 1)^{p-1}
        &=
        \ordprodsub{s \le i < t\\ 1 \le j \le r_i}
            \sum_{m=0}^{p-1} \binom{p-1}{m} (-1)^{p - 1 - m} \gij^m
        \\
        &=
        \ordprodsub{s \le i < t\\ 1 \le j \le r_i}
            \sum_{m=0}^{p-1} (-1)^m         (-1)^{p - 1 - m} \gij^m
        \\
        &=
        \ordprodsub{s \le i < t\\ 1 \le j \le r_i}
            \sum_{m=0}^{p-1} \gij^m
        =
        D_s^+.
        \qedhere
    \end{align*}
\end{proof}

\begin{lem}
    \label{lem: dimension ideal}
    \begin{equation*}
        FG \cdot D_s^+
        =
        \bigoplus
        F
        \ordprodsub{1 \le i < s\\ 1 \le j \le r_i} (\gij - 1)^\mij
        \ordprodsub{s \le i < t\\ 1 \le j \le r_i} (\gij - 1)^{p-1}
    \end{equation*}
    where the direct sum is taken for all integers \( 0 \le \mij < p \).
\end{lem}

\begin{proof}
    First, note that the dimensions are equal because
    \begin{align*}
        &
        \dim \bigoplus
        F
        \ordprodsub{1 \le i < s\\ 1 \le j \le r_i} (\gij - 1)^\mij
        \ordprodsub{s \le i < t\\ 1 \le j \le r_i} (\gij - 1)^{p-1}
        \\
        &=
        \prod_{1 \le i < s} p^{r_i}
        =
        \prod_{1 \le i < s} \order{D_i/D_{i+1}}
        =
        \order{G/D_s}
        =
        \dim FG \cdot D_s^+.
    \end{align*}
    By Lemma~\ref{lem: divisors trick}, we have
    \begin{equation*}
        FG \cdot D_s^+
        \supseteq
        \bigoplus
        F
        \ordprodsub{1 \le i < s\\ 1 \le j \le r_i} (\gij - 1)^\mij
        \ordprodsub{s \le i < t\\ 1 \le j \le r_i} (\gij - 1)^{p-1},
    \end{equation*}
    and the proof is complete.
\end{proof}

\begin{lem}
    \label{lem: powerful}
    If \( G \) is powerful then \( D_2 = D_p \).
\end{lem}

\begin{proof}
    The proof follows from Theorem~\ref{thm: Jennings-Brauer}\ref{item: M-series}.
\end{proof}

\subsection{Proof of main theorems}

\begin{proof}[{Proof of Theorem~\textup{\ref{thm: main}}}]
    The first part follows from
    \begin{gather} \begin{aligned}
        \label{eq: central}
        \z{FG}
            &\supseteq
            FG \cdot D_s^+
                & & \text{(By Lemma~\ref{lem: central ideal}.)}
            \\
            &=
            \bigoplus F
            \ordprodsub{1 \le i < s\\ 1 \le j \le r_i} (\gij - 1)^\mij
            \ordprodsub{s \le i < t\\ 1 \le j \le r_i} (\gij - 1)^{p-1}
                & & \text{(By Lemma~\ref{lem: dimension ideal}.)}.
    \end{aligned} \end{gather}

    The second part follows from
    \eqref{eq: central} and Theorem~\ref{thm: Jennings-Brauer}\ref{item: soc}.
\end{proof}

\begin{proof}[{Proof of Theorem~\textup{\ref{thm: powerful}}}]
    By Theorem~\ref{thm: main}, it remains to prove
    \begin{equation*}
        \zs{n}{FG}
        \subseteq
        \bigoplus F
            \ordprod_                {1 \le j \le r_1} (g_{1j} - 1)^{m_{1j}}
            \ordprodsub{2 \le i < t\\ 1 \le j \le r_i} (\gij   - 1)^{p-1}.
    \end{equation*}
    We claim that
    \begin{equation}
        \label{eq: claim}
        \soc{n}{FG}
        \subseteq
        \bigoplus F
            \ordprod_                {1 \le j \le r_1} (g_{1j} - 1)^{m_{1j}}
            \ordprodsub{2 \le i < t\\ 1 \le j \le r_i} (\gij   - 1)^{p-1}
    \end{equation}
    where the direct sum is taken for all integers \( 0 \le m_{1j} < p \) satisfying
    \(
        \displaystyle
        \sum_{1 \le j \le r_1} (p - 1 - m_{1j}) < n
    \).
    Let
    \begin{equation*}
        z  = \ordprodsub{1 \le i < t\\ 1 \le j \le r_i} (\gij - 1)^\mij
    \end{equation*}
    be an element of the Jennings basis lying in \soc{n}{FG}.
    Then we have the following.
    \begin{gather*} \begin{aligned}
        p
        &\ge
        n
        & &
        \\
        &>
        \sumsub{1 \le i < s\\ 1 \le j \le r_i} i(p - 1 - \mij)
        & & \text{(By Theorem~\ref{thm: Jennings-Brauer}\ref{item: soc}.)}
        \\
        &=
        \sum_{1 \le j \le r_1} (p - 1 - m_{1j})
        +
        \sumsub{p \le i \le t\\ 1 \le j \le r_i} i(\underbrace{p - 1 - \mij}_{ {} = 0})
        & & \text{(By Lemma~\ref{lem: powerful}.)}
    \end{aligned} \end{gather*}
    Hence we can express
    \begin{equation*}
        z
        =
        \ordprod_                {1 \le j \le r_1} (g_{1j} - 1)^{m_{1j}}
        \ordprodsub{2 \le i < t\\ 1 \le j \le r_i} (\gij   - 1)^{p-1}
    \end{equation*}
    as we claimed.
\end{proof}

\begin{proof}[{Proof of Corollary~\textup{\ref{cor: soc p}}}]
    The proof follows from the claim \eqref{eq: claim} and Theorem~\ref{thm: main}.
\end{proof}

\section{Examples}
\label{sec: examples}

In this section, we illustrate our results by the first non-trivial examples:
group algebras of extra-special \( p \)-groups of order \( p^3 \) for odd prime \( p \).

\subsection{Extra-special \( p \)-group \( p_+^{1+2} \)}
\label{sec: exponent p}
Let \( G \) be an extra-special \( p \)-group
of order \( p^3 \) and exponent \( p \) defined by
\begin{equation*}
    G := p_+^{1+2} =
    \langle\, a, b, c \mid a^p = b^p = c^p = [a, c] = [b, c] = 1,\ [b, a] = c \,\rangle
\end{equation*}
and set
\( x := a - 1 \),
\( y := b - 1 \), and
\( z := c - 1 \).
Then \( G \) is not powerful and the dimension subgroups of \( G \) are
\begin{equation*} 
    D_1 = \langle a, b, c \rangle,
    \quad
    D_2 = \langle c \rangle,
    \quad
    D_3 = 1.
\end{equation*}
We can show the following.
\begin{align*}
    \z{FG} &=
        \bigoplus_{\mathclap{0 \le i, j < p}}
            Fx^iy^jz^{p-1}
        \oplus
        \bigoplus_{\mathclap{0 \le k < p - 1}}
            Fz^k
    \\
    \soc{n}{FG} &=
        \bigoplus_{\mathclap{\substack{0 \le i, j, k < p\\ 4(p-1) - (i + j + 2k) < n}}}
            Fx^iy^jz^k
    \\
    \intertext{Hence we have}
    \zs{n}{FG} &=
        \bigoplus_{\mathclap{\substack{0 \le i, j    < p\\      2(p-1) - (i + j) < n}}}
            Fx^iy^jz^{p-1}
        \oplus
        \bigoplus_{\mathclap{\substack{0 \le k < p-1    \\           4(p-1) - 2k < n}}}
            Fz^k.
\end{align*}

\subsection{Extra-special \( p \)-group \( p_-^{1+2} \)}
\label{sec: exponent p^2}
Let \( G \) be an extra-special \( p \)-group
of order \( p^3 \) and exponent \( p^2 \) defined by
\begin{equation*}
    G := p_-^{1+2} =
    \langle\, a, b \mid a^p = b^{p^2} = 1,\ b^a = b^{1 + p} \,\rangle
\end{equation*}
and set
\( x := a - 1 \),
\( y := b - 1 \), and
\( z := c - 1 \) where \( c = b^p \).
Then \( G \) is powerful and the dimension subgroups of \( G \) are
\begin{equation*}
    D_1 = \langle a, b, c \rangle,
    \quad
    D_2 = \dotsb = D_p = \langle c \rangle,
    \quad
    D_{p+1} = 1.
\end{equation*}
We can show the following.
\begin{align*}
    \z{FG} &=
        \bigoplus_{\mathclap{0 \le i, j < p}}
            Fx^iy^jz^{p-1}
        \oplus
        \bigoplus_{\mathclap{0 \le k < p - 1}}
            Fz^k
    \\
    \soc{n}{FG} &=
        \bigoplus_{\mathclap{\substack{0 \le i, j, k < p\\ (p + 2)(p - 1) - (i + j + pk) < n}}}
            Fx^iy^jz^k
    \intertext{Hence we have}
    \zs{n}{FG} &=
        \bigoplus_{\mathclap{\substack{0 \le i, j    < p\\            2(p - 1) - (i + j) < n}}}
            Fx^iy^jz^{p-1}
        \oplus
        \bigoplus_{\mathclap{\substack{0 \le k < p - 1  \\           (p + 2)(p - 1) - pk < n}}}
            Fz^k.
\end{align*}
In particular, we have
\begin{equation*}
    \soc{p}{FG}
        =
        \bigoplus_{\mathclap{\substack{0 \le i, j    < p\\                   i + j \ge p - 1}}}
            Fx^iy^jz^{p-1}
        \subseteq
        \z{FG}
\end{equation*}
as expected.
Note \( x^{p-1}y^{p-1}z^{p-2} \in \soc{p+1}{FG} \setminus \z{FG} \).

\subsection{Remark}
In the above examples, there are two series of elements of \zs{n}{FG}:
Of the form \( x^iy^jz^{p-1} \) and \( z^k \).
Our main theorem states that an element of the Jennings basis like \( x^iy^jz^{p-1} \) is always central for all \( p \)-groups.

\section*{Acknowledgments}
The author wishes to acknowledge Professor Shigeo Koshitani and Dr. Yoshihiro Otokita for helpful discussions.

\end{document}